\documentclass[12pt]{amsart}
\usepackage{amsmath, amsthm, amscd, amsfonts, amssymb, graphicx, color}
\usepackage[bookmarksnumbered, colorlinks, plainpages]{hyperref}

\newtheorem{theorem}{Theorem}[section]

\newtheorem{proposition}[theorem]{Proposition}

\theoremstyle{definition}
\newtheorem{definition}[theorem]{Definition}
\newtheorem{example}[theorem]{Example}

\theoremstyle{remark}
\newtheorem{remark}[theorem]{Remark}
\numberwithin{equation}{section}

     \newcommand{\lra}{\longrightarrow} \newcommand{\ra}{\rightarrow}
 \newcommand{\h}{\mathcal{H}}
  \renewcommand{\H}{\mathcal{H}}
  \newcommand{\D}{\mathcal{D}}

 \newcommand{\F}{\mathcal{F}}
  \renewcommand{\L}{\mathcal{L}}

 \newcommand{\fN}{\mathfrak{N}}

 \newcommand{\dom}{\mathrm{dom}}
  \newcommand{\ran}{\mathrm{ran}}
 \newcommand{\spec}{\mathrm{spec}}
 \renewcommand{\ker}{\operatorname{ker}}

 \newcommand{\R}{\mathbb{R}}

 \newcommand{\N}{\mathbb{N}}

  \newcommand{\lN}{{\ell}_2(\N)}
  \newcommand{\kN}{{\mathfrak{k}}_2(\N)}

\renewcommand{\Re}{\mathfrak{K}}
\newcommand{\hsp}{{\hspace{-1pt}}}
\newcommand{\hs}{{\hspace{1pt}}}
\newcommand{\ip}[2][\cdot\hs]{\langle #1,#2\rangle}
\newcommand{\id}{\mathrm{id}}
\newcommand{\dd}{\mathrm{d}}

\renewcommand{\[}{\begin{equation}}
\renewcommand{\]}{\end{equation}}
\newcommand{\wegengruen}{\end{equation}}

\begin{document}
\setcounter{page}{1}

\title[Frames in Krein spaces arising  from a $W$-metric]
{Frames in Krein spaces arising from a non-regular $W$-metric}

\author[K. Esmeral, O. Ferrer, E. Wagner]{Kevin Esmeral$^1$, Osmin Ferrer$^{2}$ and Elmar Wagner$^3$}

\address{$^{1}$ Cinvestav,  Instituto Polit\'ecnico Nacional, 
Av.Instituto Polit\'ecnico Nacional 2508, Col. San Pedro Zacatenco, CP 07360, D.F., 
M\'exico.}
\email{matematikoua@gmail.com}

\address{$^{2}$  Departamento de Matem\'aticas,  Universidad Surcolombiana, Neiva, Huila, Colombia.}
\email{osmin.ferrer@usco.edu.co}

\address{$^{3}$ Instituto de F\'isica y Matem\'aticas,
Universidad Michoacana de San Nicol\'as de Hidalgo. Edificio C-3, Ciudad Universitaria.
C. P. 58040 Morelia, Michoac\'an, M\'exico}
\email{elmar@ifm.umich.mx}





\begin{abstract}
A definition of frames in Krein spaces is stated and a 
complete characterization is given by comparing them 
to frames in the associated Hilbert space. The basic tools 
of frame theory are described in the formalism of Krein spaces. 
It is shown how to transfer a frame for Hilbert spaces to 
Krein spaces given by a  $W$-metric, where 
the Gram operator $W$ is not necessarily regular 
and possibly unbounded. 
\end{abstract} 

\maketitle

\section{Introduction}
The theory of frames for Hilbert spaces originates from \cite{DS} and was further developed in \cite{Dau,DGM}. 
Frames can be thought of as ``over-complete bases" and their over-completeness makes them more flexible than 
orthonormal bases. They proved to be a powerful tool, e.g., in signal processing and 
wavelet analysis \cite{G}. It is only natural that one wants to have the same tools available for Krein spaces. 
This, of course, can be done without any changes by considering the associated Hilbert space. 
However we prefer a more direct approach by taking the structure of Krein spaces into account, 
avoiding thus the switching between Krein and Hilbert spaces. 
All results of the present paper are part of the first author's master thesis \cite{E}.
Similar ideas have been developed independently in  \cite{GMMM} and~\cite{PW}. 

Our approach to the theory of frames for Krein spaces is presented in Section~\ref{FiKS}. 
There we give a definition of frames for Krein spaces by replacing the 
positive definite inner product in the definition of a frame for a Hilbert space by 
a (possibly) indefinite inner product. As expected, we prove in Theorem \ref{prop1} that 
the theory of frames for a Krein space and 
the theory of frames for the associated Hilbert space 
are equivalent. 
Then we reformulate the basic tools of frame theory in the language of Krein space. 
For instance, the pre-frame operator is allowed to have a Krein space as its domain. 
Our definitions are such that the frame operator (Definition \ref{defS}) and the 
Frame Decomposition Theorem (Theorem \ref{THM}) are exactly as in the Hilbert space case with 
the positive definite inner product replaced by the inner product of the Krein space. 
Furthermore, we discuss canonical dual frames (Proposition \ref{DF}), tight frames and their relation 
to $J$-orthonormalized bases (Proposition \ref{tf}), and the relation between frames 
and orthogonal projections commuting with the fundamental symmetry (Proposition~\ref{link}). 

In Section \ref{Sec-W}, we apply our theory to Krein spaces given by 
a non-regular or/and unbounded Gram operator $W$ acting on a Hilbert space $\h$. 
The fundamental observation (Proposition~\ref{not}) is that, 
whenever the Gram operator $W$ is non-regular or unbounded, a frame for $\h$ 
can never be a frame for the Krein space $\h_W$ constructed from $W$. 
Nevertheless, we show how to transfer any frame for the Hilbert space $\h$ to a frame 
for the Krein space $\h_W$. The basic idea is to construct a unitary operator 
between $\h_W$ and $\h$ starting from the square root $\sqrt{|W|}$. 
For a better understanding, we distinguish between the cases 
where $W$ is bounded (Theorem~\ref{Wbounded}), and 
where $W$ is unbounded but  \,$0\notin \spec(W)$ (Theorem~\ref{Wunbounded}). 
The general situation is treated in Theorem~\ref{V3}. Note that these results also 
apply to positive Gram operators $W$ when $\h_W$ is actually a 
Hilbert space.

\section{Preliminaries} 

The purpose of this section is to fix notations and to recall the basic elements of frame theory. 
For more details on Krein spaces, we recommend \cite{A} and \cite{B}. 
A comprehensive introduction to frame theory can be found in \cite{Chr}. 

\subsection{Krein spaces} 
\label{KS}
Throughout this paper, $(\Re, [\cdot,\cdot])$ denotes a Krein space 
with fundamental decomposition $\Re_+\oplus\Re_-$ and 
fundamental symmetry $J$ given by 
\[ \label{J}
 J (k^+ + k^-) = k^+ - k^-,\qquad  k^+ + k^-\in \Re_+\oplus\Re_-, 
\]
such that the $J$-inner product
\[ \label{ipJ}
[  h^+ \hsp +\hsp  h^-\hs, k^+\hsp  +\hsp  k^-]_J:= [ h^+\hsp  +\hsp  h^-\hs, J (k^+\hsp  +\hsp  k^-)] 
= [  h^+ \hs, k^+ ] -[   h^-\hs,  k^-], \ \ h^\pm,k^\pm\hsp \in\hsp \Re_\pm, 
\]
turns $(\Re, [\cdot,\cdot]_J)$ into a Hilbert space. 
The positive definite inner product $[\cdot,\cdot]_J$ defines a topology on $\Re$ by the $J$-norm 
$$
\|k\|_{J}:= \sqrt{[k\hs , k]_J} =  \sqrt{[k\hs , Jk]},\qquad k\in\Re, 
$$
and  $\Re_+\oplus\Re_-$ becomes the orthogonal sum of Hilbert spaces. 
Note that $J^2=1$ by \eqref{J}.

Topological concepts     
such as convergence and bounded linear operators refer to the topology induced by 
the Hilbert space norm $\|\cdot\|_{J}$. 
The angle brackets  $\ip[\cdot]{\cdot}$ are reserved to denote the 
positive definite inner product of a given Hilbert space $\h$.  
When no confusion can arise, we write shortly $\Re$ for a Krein space $(\Re, [\cdot,\cdot])$,  
and $\h$ for a Hilbert space $(\h,\ip[\cdot]{\cdot})$. 
The symbol  $\kN$ will be used if $\Re = \lN$ as a complex vector space. 
A $J$-orthonormal basis is a complete family $\{e_n\}_{n\in\fN}\subset \Re$ 
such that $[e_n,e_m]=0$ for $n\neq m$ and $|[e_n,e_n]|=1$. 
A subspace $V\subset \Re$ is said to be uniformly $J$-positive (resp.\ uniformly $J$-negative) 
if there exists an $\varepsilon >0$ such that 
$[v,v]\geq \varepsilon \|v\|_J^2$ (resp.\ $-[v,v]\geq \varepsilon \|v\|_J^2$) for all $v\in V$.

The unique adjoint $T^* : (\Re_2, [\cdot,\cdot]_2)\ra (\Re_1, [\cdot,\cdot]_1)$ of a 
bounded linear operator $T: (\Re_1, [\cdot,\cdot]_1)\ra (\Re_2, [\cdot,\cdot]_2)$ 
is always taken with respect to the specified inner products, i.e., 
$$
[T^* h,k  ]_1 = [h, T k]_2\quad \text{for \,all}\ \ k\in\Re_1,\ \,h\in\Re_2. 
$$
For instance, it follows from \eqref{J} and \eqref{ipJ} that 
$J: (\Re, [\cdot,\cdot]) \ra (\Re, [\cdot,\cdot])$ and 
$J: (\Re, [\cdot,\cdot]_J) \ra (\Re, [\cdot,\cdot]_J)$
are self-adjoint, i.e.,  $J^*=J$. 
Furthermore, 
the identity operator 
\[ \label{I}
I_J : (\Re, [\cdot,\cdot])\lra (\Re, [\cdot,\cdot]_J),\qquad I_J\hs k=k, 
\]
has 
the adjoint $J: (\Re, [\cdot,\cdot]_J)\ra (\Re, [\cdot,\cdot])$ since 
$[I_J \hs k,h]_J = [k, h]_J =[k,J\hs  h]$
for all  $h,k\in\Re$. A self-adjoint operator $A=A^*$ on 
$(\Re, [\cdot,\cdot])$ is called uniformly positive, if  
$[k,Ak] \geq \varepsilon [k, k]_J$ for a suitable constant $\varepsilon >0$ 
and all $k\in\Re$.  Equivalently, since $[k,Ak]=[k,JAk]_J$, 
we have $JA\geq \varepsilon$ on the Hilbert space $(\Re, [\cdot,\cdot]_J)$. 
As a consequence, $A$ has a bounded inverse.

The fundamental projections 
\[  \label{P}
P_+:=\mbox{$\frac{1}{2}$}(1+J),\qquad P_-:= \mbox{$\frac{1}{2}$}(1-J)
\]
act on $\Re=\Re_+\oplus\Re_-$ by $P_+(k^+\hsp +\hsp k^-) = k^+$ and 
$P_-(k^+\hsp +\hsp k^-) = k^-$. Equation \eqref{P} implies immediately that  $P_\pm$ 
and  $J$ commute. Moreover, $P_+$ and $P_-$ are orthogonal projections, 
i.e.\ $P_{\pm}^2= P_{\pm}= P_{\pm}^*$, regardless of whether we 
consider $[\cdot,\cdot]$ or $[\cdot,\cdot]_J$ on $\Re$. 

We close this subsection with a short review of  
Krein spaces given by a \emph{regular} Gram operator $W$ \cite{A}. 
The case when $W$ is unbounded or $0\in\spec(W)$  will be discussed in 
Section \ref{Sec-W}.  Let thus $W$ be a bounded self-adjoint operator 
on a Hilbert space $(\h,\ip{\cdot})$ such that $0\notin\spec(W)$. 
Define a non-degenerate inner product on $\h$ by 
\[ \label{ipW}
        [f,g]:= \ip[f]{W\hs g},\quad f,\, g \in\h. 
\] 
Since $W^*= W$ and $0\notin\spec(W)$, the operator $J$ given by the 
polar decomposition $W=J\hs |W|$ is self-adjoint and unitary, i.e.\ $J^* = J$ and $J^2 = 1$. Hence 
the orthogonal projections $P_+$ and $P_-=1-P_+$ from \eqref{P} 
project onto the eigenspaces corresponding to the 
eigenvalues $1$ and $-1$ of $J$, respectively. As a consequence, 
$J$ acts on the orthogonal sum  
$\h=P_+\h \oplus P_-\h$ by  
$J (h^+ \hsp + \hsp  h^-) = h^+ \hsp  - \hsp  h^-$, where $h^\pm\in P_\pm\h$. 
From $J^2 = 1$, we get $[f,g]_J := \ip[f]{JWg}= \ip[f]{|W|g}$ for all $f,g\in\h$. 
Moreover, 
the condition $0\notin\spec(W)$ implies that $\varepsilon\leq |W|\leq \| W\|$ 
for some $\varepsilon >0$, which means that 
$\varepsilon\ip[h]{h}\leq \ip[h]{|W|\hs h} \leq \| W\|\ip[h]{h}$ for all $h\in\h$.  
Therefore, the Hilbert space norm $\sqrt{\ip{\cdot}}$  and the $J$-norm $\|\cdot\|_J$ 
are equivalent. In particular, $\h=P_+\h \oplus P_-\h$ remains an orthogonal sum 
of Hilbert spaces with respect to $[\cdot,\cdot]_J$. 
Summarizing, we have shown that $(\h, [\cdot,\cdot])$ is a Krein space with 
fundamental decomposition $\h=P_+\h \oplus P_-\h$ and 
fundamental symmetry $J$ such that $\ip{\cdot}$ and $[\cdot,\cdot]_J$ 
define equivalent norms. 

\subsection{Frames in Hilbert spaces} 
\label{FiHS}
 Let $\h$ be an (infinite-dimensional) Hilbert space with inner product $\ip{\cdot}$. 
 A frame for $\h$ is a 
 countable sequence $\left\{f_{n}\right\}_{n\in \N}$ in $\h$ 
 such that there exist constants $0< A\leq B<\infty$ satisfying 
\begin{equation}                \label{fH}
A\|f\|^{2}
\,\leq\, \sum_{n\in \N}\left|\ip[f_{n}]{f}\right|^{2}
\,\leq\, B\|f\|^{2}\quad \text{for all} \ \, f\in\h.  
\end{equation}
Given a frame $\{f_n\}_{n\in\N}$ for $\h$, 
one defines the pre-frame operator by 
\[   \label{pF}
T_0: \lN\, \lra \, \h,\quad T_0\hs (\alpha_n)_{n\in\N} = \sum_{n\in\N} \alpha_n f_n\hs .
\]
It follows from the Bessel condition $\sum_{n\in \N}\left|\ip[f_{n}]{f}\right|^{2}\leq B\|f\|^{2}$ 
that $T_0$ is well defined and bounded. A direct computation shows that 
\[ \label{apF}
T_0^* :\h\,\lra\,  \lN, \quad T_0^*(f) = (\ip[f_n]{f})_{n\in\N} 
\]
yields its adjoint. 
The frame operator $S$ is defined by $S:=T_0\hs T_0^*$. 
By Equations \eqref{pF} and \eqref{apF}, 
\[  \label{Sh}
Sf =\sum_{n\in\N} \ip[f_n]{f} f_n, \quad f\in\h. 
\]
The inequalities in \eqref{fH} imply that $A\leq S\leq B$, thus $S$ has bounded inverse. 
Inserting \eqref{Sh} into $f = SS^{-1}f= S^{-1} S f$ and using $(S^{-1})^*= S^{-1}$,  
we arrive at the so-called Frame Decomposition Theorem: 
\[ \label{FDT}
f= \sum_{n\in \N}\ip[S^{-1} f_{n}]{f} f_n = \sum_{n\in \N}\ip[ f_{n}]{f} S^{-1}f_n\quad 
\text{for all} \ \,f\in\h. 
\]
From $S^{-1}f =\sum_{n\in \N}\ip[ S^{-1}  f_{n}]{f} S^{-1}f_n$ and $B^{-1}\leq S^{-1} \leq A^{-1}$, 
it follows that $\{S^{-1} f_n\}_{n\in\N}$ is a frame for $\h$ admitting the 
frame bounds $B^{-1}\leq A^{-1}$. 
A frame $\{ g_n\}_{n\in\N}\subset \h$ satisfying 
\[\label{df}
f= \sum_{n\in \N}\ip[g_n]{f} f_n = \sum_{n\in \N}\ip[ f_{n}]{f} g_n\quad 
\text{for all} \ \,f\in\h. 
\] 
is called a dual frame of $\{f_n\}_{n\in\N}$. 
By \eqref{FDT}, $\{S^{-1} f_n\}_{n\in\N}$ 
yields an example of a dual frame of $\{f_n\}_{n\in\N}$,  
and one usually refers to it as the canonical dual frame.

\section{Frames in Krein spaces} 
\label{FiKS}
In this section we state a definition of frames in Krein spaces 
and prove that they are in one-to-one 
correspondence to frames in the associated Hilbert space. 
Some basic tools of frame theory such as 
pre-frame operator, frame operator and 
dual frame are then described in the language of 
Krein spaces. 
\begin{definition}  \label{fK}
 Let $\Re$ be a Krein space. A countable sequence $\left\{k_{n}\right\}_{n\in \fN}\subset \Re$ 
 is called a \emph{frame for} $\Re$, if there exist constants $0< A\leq B<\infty$ such that
\begin{equation} 
A\hs \|k\|_{J}^{2}
\,\leq\, \sum_{n\in \fN}\left|\left[k_{n},k\right]\right|^{2}
\,\leq\, B\hs \|k\|_{J}^{2}\quad \text{for all} \  \,k\in\Re.  \label{m1}
\end{equation}
\end{definition}
\begin{remark}
Since we are mostly interested in infinite-dimensional spaces, and 
since one can always fill up a finite frame with zero elements, 
we assume that $\fN=\N$ whenever the finiteness of $\fN$ is of no importance. 
\end{remark}

As in the Hilbert space case, we refer to $A$ and $B$ as frame bounds. 
The greatest constant $A$ and the smallest constant $B$ satisfying \eqref{m1} 
are called optimal lower frame bound and optimal upper frame bound, respectively. 
A frame is tight, if one can choose $A=B$. If a frame ceases to be a frame 
when an arbitrary element is removed, the frame is said to be exact.

The next theorem shows that frames for a Krein space are essentially the same objects as frames 
for the associated Hilbert space. 
\begin{theorem}\label{prop1}
Let $\Re$ be a Krein space and $\{k_{n}\}_{n\in\N}$ a sequence in $\Re$. 
The following statements are equivalent:\\[6pt]
\phantom{ii}i)   $\{k_{n}\}_{n\in\N}$ is a frame for the Krein space $\Re$ with frame bounds $A\leq B$.\\[6pt]
\phantom{i}ii) $\{Jk_{n}\}_{n\in\N}$ is a frame for the  Krein space $\Re$ with frame bounds $A\leq B$.\\[6pt]
iii) $\{k_{n}\}_{n\in\N}$ is a frame for the  Hilbert space $\left(\Re, [\cdot,\cdot]_{J}\right)$ with frame bounds $A\leq B$.\\[6pt]
\phantom{}\hs\hs iv) $\{Jk_{n}\}_{n\in\N}$ is a frame for the  Hilbert space $\left(\Re, [\cdot,\cdot]_{J}\right)$ with frame bounds $A\leq B$.
\end{theorem}
\begin{proof} The equivalence of i) and iv) follows from 
$$\left|\left[k,k_{n}\right]\right|^{2}=\left|\left[ Jk,Jk_{n}\right]\right|^{2}=\left|\left[k,Jk_{n}\right]_J\right|^{2}.$$ 
The same argument applied to   $\{Jk_{n}\}_{n\in\N}$ together with $J^2=1$
proves the equivalence of ii) and iii). 
Since $J$ is a unitary operator on  $\left(\Re, [\cdot,\cdot]_{J}\right)$, the equivalence 
iii) and iv) is obvious. This finishes the proof. 
\end{proof} 

Recall that, given a frame $\{f_n\}_{n\in\N}$ for a Hilbert space $\h$, the pre-frame operator 
$T_0: \lN \ra \h$
is defined by \eqref{pF}. 
In the special case when  $\h=\lN$ and 
$\{f_n\}_{n\in\N}$ coincides with  the standard basis of $\lN$,  $T_0$ is just the identity. 
By Theorem \ref{prop1}, we could define a pre-frame operator $T$ for frames in a Krein space in the same way. 
But then, if the Hilbert space $\h$ is replaced by a non-trivial Krein space  of sequences $\kN$, 
the pre-frame operator $T:\lN\ra \kN$ can never be the identity operator in the strict sense. 
Therefore we also allow non-trivial Krein spaces
for the domain of the pre-frame operator. 
\begin{definition}
 Let $(\Re, [\cdot,\cdot])$ be a Krein space with fundamental symmetry $J$ and let 
 $(\kN, [\cdot,\cdot])$ be a Krein space with fundamental symmetry $\tilde J$ 
 such that $[\cdot,\cdot]_{\hsp{\tilde J}}$ agrees with the standard inner product  $\ip{\cdot}$ on $\lN$. 
 Given a frame $\{k_{n}\}_{n\in\N}$ for $\Re$, 
 the linear map 
 \[ \label{T}
 T: \kN \lra \Re,\quad T\hs (\alpha_n)_{n\in\N} = \sum_{n\in\N} \alpha_n k_n
 \]
 is called \emph{pre-frame operator}. 
\end{definition}

Note that $T$ is well defined and bounded since it factorizes as 
\[ \label{fac}
 (\kN, [\cdot,\cdot]) \overset{I_{\tilde J}}{\lra} (\lN,\ip{\cdot}) \overset{T_{0}}{\lra} (\Re, [\cdot,\cdot]_J) 
 \overset{I_{J}^{-1}}{\lra} (\Re, [\cdot,\cdot]) 
\]
and all operators in \eqref{fac} are bounded. 
Here we use the fact that $T_0$ is a pre-frame operator by Theorem \ref{prop1}.iii), 
apply \cite[Theorem 5.5.1]{Chr} to $T_0$,  
and observe that $I_{\tilde J}$ and $I_{ J}$ 
defined by \eqref{I} are bijections. 
Since $T= I_{ J}^{-1}\hs T_0\hs I_{\tilde J}$\hs, it follows from 
the Hilbert space frame theory applied to $T_0$  that a sequence $\{k_{n}\}_{n\in\N}$ 
is a frame for $\Re$, if and only if $T$ is well defined (i.e.\ bounded)  and surjective.

The adjoint of $T$ is given by 
\[  \label{T*}
T^{*}k= \tilde{J}([k_n,k])_{n\in\N}, \quad k\in\Re. 
\]
Indeed, for all $(\alpha_n)_{n\in\N}\in \kN$ and $k\in\Re$, one has  
$$
[T(\alpha_n)_{n\in\N}, k]=\hsp \sum_{n\in\N} \bar\alpha_n [k_n,\hsp k] 
= \ip[(\alpha_n)_{n\in\N}]{([k_n,\hsp k])_{n\in\N}} 
=[(\alpha_n)_{n\in\N}, \tilde{J}([k_n,\hsp k])_{n\in\N}]. 
$$
To obtain a frame operator given by a formula analogous to \eqref{Sh} with $\ip{\cdot}$ replaced 
by $[\cdot,\cdot]$, we make the following definition: 
\begin{definition} \label{defS}
Let $(\Re, [\cdot,\cdot])$ be a Krein space with fundamental symmetry $J$, 
 $(\kN, [\cdot,\cdot])$ a Krein space with fundamental symmetry $\tilde J$ 
 such that $[\cdot,\cdot]_{\tilde J}$ agrees with the standard inner product  $\ip{\cdot}$ on $\lN$,  
 and $\{k_{n}\}_{n\in\N}\subset\Re$  a frame for $\Re$. 
The operator 
\[  \label{S}
S:= T\hs\tilde J\hs T^{*}
\] 
is called \emph{frame operator}. 
\end{definition}
It follows immediately from \eqref{T}, \eqref{T*} and $\tilde{J}^2= \id$ that 
\[    \label{Sk}
 Sk = \sum_{n\in\N} [k_n\hs,k]\hs k_n,\quad k\in\Re,         
\]
as desired. 
Moreover, $S$ is clearly self-adjoint. 
If $(\kN, [\cdot,\cdot])= (\lN, \ip{\cdot})$, then $S=TT^*$, exactly as in the 
Hilbert space case. 

Equation \eqref{Sk} yields  $[k,Sk]= \sum_{n\in\N} |[k_n\hs,k]|^2$ for all $k\in\Re$. 
Replacing $k$ by $Jk$ and writing $[k_n\hs,Jk]=[k_n\hs,k]_J$, 
we get from Equation \eqref{m1} and Theorem \ref{prop1}.iii) that 
$$
A\|k\|_{J}^{2}\leq   [k\hs,SJk]_J         \leq B\|k\|_{J}^{2}, \quad k\in\Re.
$$
Hence $SJ$ is a strictly positive operator on 
the  Hilbert space $\left(\Re, [\cdot,\cdot]_{J}\right)$ satisfying 
$A \leq SJ \leq B$. Therefore $S$ is invertible with bounded inverse, and 
$$
B^{-1}\leq JS^{-1} \leq A^{-1}, 
$$
where we used $J^{-1}=J$. Equivalently, when viewed as positive operators 
on the Krein space $\left(\Re, [\cdot,\cdot]\right)$, we have $A J\leq S \leq BJ$ 
and 
$$
B^{-1}J\leq S^{-1} \leq A^{-1}J. 
$$

The invertibility of the operator $S$ 
allows us to state the 
Frame Decomposition Theorem -- the most important theorem 
of frame theory -- exactly as in the Hilbert space case, 
cf.\ \cite[Theorem 5.1.6]{Chr}.

\begin{theorem}[Frame Decomposition Theorem] \label{THM}
Let $\{k_{n}\}_{n\in\N}$ be a frame for the Krein space $\Re$,  
and let $S$ be defined as in Definition \ref{defS}. Then 
\begin{align}
k &= \sum_{n\in\N} [k_n\hs,k]\hs S^{-1} k_n, \label{FDT1} \\
k &= \sum_{n\in\N} [S^{-1} k_n\hs,k]\hs  k_n,  \label{FDT2} 
\end{align}
and both series converge unconditionally for all $k\in\Re$. 
\end{theorem}
\begin{proof}
Equation \eqref{FDT1} is proved by applying \eqref{Sk} to $k=S^{-1}Sk$. 
Since the self-adjointness of $S$ implies the self-adjointness of $S^{-1}$, 
Equation \eqref{FDT2} follows from \eqref{Sk} applied to $k=S S^{-1}k$. 
For the proof of the unconditional convergence, we replace again $k$ by $Jk$, 
write $[k_n\hs,Jk]=[k_n\hs,k]_J$, invoke Theorem \ref{prop1}.iii), 
and refer to the Hilbert space case \cite[Theorem 5.1.6]{Chr}. 
\end{proof}

By Theorem \ref{prop1}, each frame $\{k_{n}\}_{n\in\N}$ for $(\Re, [\cdot,\cdot])$ gives rise to 
three other frames with slightly different frame operators. 
In the following, we will relate these frame operators to $S$ from \eqref{Sk}. 
First, consider the frame $\{Jk_{n}\}_{n\in\N}$ for $(\Re, [\cdot,\cdot])$. 
Denoting the corresponding frame operator  by $S_0$, 
we get from \eqref{Sk}
$$
S_{ 0}\hs k = \sum_{n\in\N} [Jk_n\hs,k]\hs Jk_n
= \sum_{n\in\N} [k_n\hs,Jk]\hs Jk_n,\quad k\in\Re.  
$$
Comparing this equation with \eqref{Sk} shows that the two frame operators 
are related by 
$$
S_{0} = J\hs S\hs J. 
$$
Next, let $S_1$ be the frame operator of the frame $\{k_{n}\}_{n\in\N}$ for $(\Re, [\cdot,\cdot]_J)$. 
Then, by Equation \eqref{Sh}, we have 
$$
S_1\hs k = \sum_{n\in\N} [k_n\hs,k]_J\hs k_n = \sum_{n\in\N} [k_n\hs,J k]\hs k_n, \quad k\in\Re,  
$$
and thus 
$$
 S_1= S\hs J. 
$$
Finally, with $S_2$ denoting the frame operator of the frame 
$\{Jk_{n}\}_{n\in\N}$ for $(\Re, [\cdot,\cdot]_J)$, we get from \eqref{Sh} 
$$
S_2k = \sum_{n\in\N} [Jk_n\hs,k]_J\hs J k_n = \sum_{n\in\N} [k_n\hs,k]\hs J k_n, \quad k\in\Re,  
$$
so that 
$$
S_2 = J \hs S. 
$$

Recall that, by Equation \eqref{df}, a dual frame of the frame 
$\{k_n\}_{n\in\N}$ for the Hilbert space $(\Re, [\cdot,\cdot]_J)$
is a frame $\{g_n\}_{n\in\N}$ for $(\Re, [\cdot,\cdot]_J)$ 
satisfying 
$$
k = \sum_{n\in\N} [g_n\hs,k]_J\hs  k_n,  \quad \text{for all}\ k\in\Re. 
$$
We will now state an analogous definition for Krein spaces and then 
describe dual frames 
in terms of the fundamental symmetry $J$ and the frame operator $S$. 
These dual  frames are called \emph{canonical} dual frames. 

\begin{definition}
Let $\{k_{n}\}_{n\in\N}$ be a frame for the Krein space $\Re$. 
A frame $\{h_n\}_{n\in\N}$ for $\Re$ is called a 
\emph{dual frame} of $\{k_{n}\}_{n\in\N}$ if 
$$
k = \sum_{n\in\N} [h_n\hs,k]\hs  k_n,  
$$
for all $k\in\Re$. 
\end{definition}
\begin{proposition} \label{DF}
Let $S$ be the frame operator of the frame 
$\{k_{n}\}_{n\in\N}$  for the Krein space $\Re$. 
Then 
\begin{itemize}
\item[\textit{i)}] $\{S^{-1}k_n\}$ is a dual frame of $\{k_{n}\}_{n\in\N}$  for the Krein space $\Re$. 
\item[\textit{ii)}] $\{JS^{-1}k_n\}$ is a dual frame of $\{Jk_{n}\}_{n\in\N}$  for the Krein space $\Re$. 
\item[\textit{iii)}] $\{JS^{-1}k_n\}$ is a dual frame of $\{k_{n}\}_{n\in\N}$  for the Hilbert space $\left(\Re, [\cdot,\cdot]_{J}\right)$. 
\item[\textit{iv)}] $\{S^{-1}k_n\}$ is a dual frame of $\{Jk_{n}\}_{n\in\N}$  for the Hilbert space $\left(\Re, [\cdot,\cdot]_{J}\right)$. 
\end{itemize}
If $0<A\leq B<\infty$ are frame constants for $\{k_{n}\}_{n\in\N}$, 
then all these dual frames admit the frame constants $0< B^{-1}\leq A^{-1}<\infty$. 
\end{proposition}
\begin{proof}
First, since $S_1=S J$ is the frame operator of  $\{k_{n}\}_{n\in\N}$  for  
$\left(\Re, [\cdot,\cdot]_{J}\right)$ 
with inverse  $S_1^{-1}\hsp=\hsp J S^{-1}$, 
it follows from the Hilbert space theory \cite[Lemma 5.1.5]{Chr} that  
$\{S_1^{-1}k_{n}\}_{n\in\N}= \{J S^{-1} k_{n}\}_{n\in\N}$ is a dual frame of $\{k_{n}\}_{n\in\N}$ 
for $\left(\Re, [\cdot,\cdot]_{J}\right)$ 
with frame constants $0< B^{-1}\leq A^{-1}<\infty$. This proves iii). Moreover, 
by Theorem~\ref{prop1}, all the other dual frames given in the 
proposition are indeed frames for the corresponding spaces and admit the same frame constants. 
Now i) is a direct consequence of \eqref{FDT2}, and ii) follows by inserting the frame $\{Jk_{n}\}_{n\in\N}$ 
and the inverse frame operator $S_0^{-1}=JS^{-1}J$ into \eqref{FDT2}. Finally iii) implies iv) 
by applying the unitary transformation $J$ to the Hilbert space $\left(\Re, [\cdot,\cdot]_{J}\right)$. 
\end{proof}

Let us also discuss the situation of tight frames 
with normalized elements. The next proposition 
shows that these frames are 
actually $J$-orthonormalized bases. 
\begin{proposition} \label{tf}
Let $\{k_{n}\}_{n\in\fN}$ be a tight frame for the Krein space $\Re$ 
with frame bounds $A=B=1$. 
Assume that $|[k_n,k_n]|=1$ for all $n\in\fN$. 
Then $\{k_{n}\}_{n\in\fN}$ is a $J$-orthonormalized basis for $\Re$. 
\end{proposition} 
\begin{proof}
For $n\in\fN$, write  
$k_n= P_+ k_n +P_-k_n=: k^+_n +k^-_n\in\Re_+\oplus\Re_-$,  
where $P_+$ and $P_-$ denote the fundamental projections of \eqref{P}. 
Set $\fN_+:=\{n\in\fN : [k_n,k_n] =1\}$ and $\fN_-:=\{n\in\fN : [k_n,k_n] =-1\}$.  
Let $m\in \fN_+$.  
Then 
$$
1= [k^+_m,k^+_m] + [k^-_m,k^-_m] \leq [k^+_m,k^+_m].
$$ 
Inserting $k:=k^+_m$ and  $A=B=1$ into \eqref{m1} gives 
\[ \label{1}
[k^+_m,k^+_m] =  \left|\left[k^+_m,k^+_{m}\right]\right|^{2} + 
\sum_{n\in \fN{\setminus}\{m\}}\left|\left[k^+_m,k_{n}\right]\right|^{2} . 
\] 
If we had $[k^+_m,k^+_m]> 1$, we would get a contradiction. 
Hence $[k^+_m,k^+_m]=1$ and therefore $k_m^-=0$. 
Now \eqref{1} implies that $[k_m,k_n]=[k^+_m,k_n] =0$ 
for all $n\in\fN$, \,$n\neq m$. 
Replacing $(\Re, [\cdot\hs,\cdot])$ by $(\Re, -[\cdot\hs,\cdot])$ shows 
that the same holds for all $m\in \fN_-$. 
The rest of the proof is routine. 
\end{proof}

The following proposition establishes a link to $J$-frames as defined in \cite{GMMM}. 
We will explain this after the proof.

\begin{proposition}  \label{link}
Let $\Re$ be a Krein space with fundamental symmetry $J$, and let $P$ be an 
orthogonal projection commuting with $J$. 

If $\{k_{n}\}_{n\in\N}$ is a frame for $\Re$ with frame bounds $A\leq B$, 
then $\{Pk_{n}\}_{n\in\N}$ is a frame for $P\Re$ and 
$\{(1-P)k_{n}\}_{n\in\N}$ is a frame for $(1-P)\Re$,  both admitting the same frame bounds. 

Conversely, if $\{k_{n}^+\}_{n\in\fN_+}$ is a frame for $P\Re$ and 
$\{k_{n}^-\}_{n\in\fN_-}$ is one for $(1-P)\Re$, both with frame bounds  $A\leq B$, 
then $\{k_{n}^+\}_{n\in\fN_+}\cup \{k_{n}^-\}_{n\in\fN_-}$
is a frame for $\Re$ ad\-mit\-ting the same frame bounds. 
\end{proposition}
\begin{proof}
Since $P$ commutes with $J$, the subspaces $P\Re$ and $(1-P)\Re$ of $\Re$ 
are Krein spaces with fundamental symmetry $PJ$ and $(1-P)J$, respectively. 
Given a frame $\{k_{n}\}_{n\in\fN}$ for $\Re$ with frame bounds $A\leq B$, 
we have for all $k \in P\Re$
$$
A\| k \|_J^2 = A\| P k \|_J^2
\leq \sum_{n\in \N}\left|\left[Pk,k_{n}\right]\right|^{2}
=  \sum_{n\in \N}\left|\left[ k,P k_{n}\right]\right|^{2}
\leq B\| P k \|_J^2 = B\| k \|_J^2,
$$
hence $\{Pk_{n}\}_{n\in\N}$ is a frame for $P\Re$ with frame bounds $A\leq B$. 
The same remains true for $P$ replaced by $1-P$. 
This completes the prove of i). 

Now let $\{k_{n}^+\}_{n\in \fN_+}$ and $\{k_{n}^-\}_{n\in \fN_-}$ be two frames 
satisfying the assumptions stated in the proposition. 
For $k\in\Re$, set $k^+:= Pk$ and $k^-:=(1-P)k$. 
Note that $[k,k_{n}^+] =[k,Pk_{n}^+] =[Pk,k_{n}^+] = [k^+,k_{n}^+]$ and,  
similarly, $[k,k_{n}^-]= [k^-,k_{n}^-]$. 
From $PJ=JP$, it follows that 
$\|k\|_J^2=\|k^+\|_J^2 + \|k^-\|_J^2$. 
Therefore 
\begin{align*}
A\|k\|_J^2& \,=\,  A\|k^+\|_J^2 + A \|k^-\|_J^2\\
&\,\leq\, \sum_{n\in \fN_+} \left|\left[k^+,k_{n}^+\right]\right|^{2} + 
\sum_{n\in \fN_-} \left|\left[k^-,k_{n}^-\right]\right|^{2} 
=\sum_{n\in \fN_+} \left|\left[k,k_{n}^+\right]\right|^{2} + 
\sum_{n\in \fN_-} \left|\left[k,k_{n}^-\right]\right|^{2} \\
&\,\leq\, B\|k^+\|_J^2 + B\|k^-\|_J^2 
\,=\, B \|k\|_J^2, 
\end{align*}
which finishes the proof.
\end{proof}

The definition of a $J$-frame given in \cite{GMMM} can be rephrased as 
follows: Given a Bessel sequence $\{f_n\}_{n\in\N}$ in 
$(\Re, [\cdot,\cdot]_J)$ (i.e., 
$\{f_n\}_{n\in\N}$  satisfies the upper bound condition in \eqref{fH}), 
consider the inner product on $\kN$  given by 
$$
  [e_n,e_n] := \left\{    
\begin{array}{rc}
1 & \text{if} \ \  [f_n,f_n]\,\geq\, 0,\\[6pt]
-1 & \text{if} \ \  [f_n,f_n] \,<\, 0,
\end{array}          
\right.   
$$
where $\{e_n\}_{n\in\N}$ denotes the standard basis of the sequence space $\kN$. 
Let $\tilde J$ denote the fundamental symmetry of $(\kN, [\cdot,\cdot])$ 
such that $[\cdot,\cdot]_{\tilde J}$ becomes the standard inner product  $\ip{\cdot}$ on $\lN$. 
Then the fundamental projection  $P_+=\frac{1}{2}(1+{\tilde J})$ projects onto the closed subspace 
generated by those $e_n$ satisfying $[e_n\hs ,e_n]=1$. Now, $\{f_n\}_{n\in\N}$ is 
a  \emph{$J$-frame} if the ranges $\ran(TP_+)$ and $\ran(T(1-P_+))$ are 
maximal uniformly $J$-positive and $J$-negative subspaces of $\Re$, respectively, 
where $T$ denotes the pre-frame operator \eqref{T}. 
In \cite[Example 3.3]{GMMM}, it was shown that not every frame in $\Re$ is a $J$-frame. 
As a consequence, Definition \ref{fK} and the definition of a $J$-frame are not equivalent. 
However, by Proposition \ref{link}, 
each frame $\{f_n\}_{n\in\N}$ in $\Re$ gives rise to a $J$-frame 
by considering $\{\frac{1}{2}(1+ J) f_n\}_{n\in\N} \cup \{\frac{1}{2}(1-J)f_n\}_{n\in\N}$ 
and omitting zero elements if required.  

We close this section with a few remarks concerning examples. 
Let $(\h,\ip{\cdot})$ be a Hilbert space 
and $J$ a bounded linear operator on $\h$ satisfying $J^2= 1$ and $J^* =J$. 
Define $[h,k]:= \ip[h]{Jk}$ for all $h,k\in\h$. 
Then $(\h, [\cdot\hs,\cdot])$ is a Krein space with fundamental symmetry $J$. 
Note that, by considering the associated Hilbert space with 
inner product $[\cdot\hs,\cdot]_J$, 
each Krein space is of this type. 
By Theorem~\ref{prop1}, all frames for the Krein space $(\h, [\cdot\hs,\cdot])$ 
are obtained from frames for the Hilbert space $(\h,\ip{\cdot})$. 

Our interest in frames for Krein spaces originates from  
${L}_2$-spaces $\L_2(\Omega,\mu)$. If $\mu$ is 
a positive measure on a $\sigma$-algebra over $\Omega$ 
and $\varphi$ is a measurable real function such that 
$0< \mathrm{ess\;inf}\hs  |\varphi | \leq    \mathrm{ess\;sup}\hs |\varphi | <\infty$, 
then 
\[ \label{L2}
[f,g]:= \int \bar f \hs g\hs \varphi\hs \dd\mu,\quad f,g\in \L_2(\Omega,\mu), 
\]
defines an (indefinite) 
inner product such that $(\L_2(\Omega,\mu), [\cdot\hs,\cdot])$ 
becomes a Krein space with fundamental symmetry $J$ given by the 
multiplication by the sign function of $\varphi$. 

Now, if we allow  $\mathrm{ess\;inf}\hs  |\varphi | =0$ 
(but require without loss of generality that $\mu \{\varphi=0\}=0$), 
then $(\L_2(\Omega,\mu), [\cdot\hs,\cdot]_J)$ is not complete. 
In order to obtain a Krein space with the inner product determined by \eqref{L2}, 
one has to take a completion. 
On the dense subspace $\L_2(\Omega,\mu)$, 
the positive inner product is then given  by 
$$
[f,g]_J= \int \bar f \hs g\hs |\varphi |\hs \dd\mu. 
$$
By Theorem \ref{prop1}, we may apply the frame theory either to the 
Krein space or to the associated Hilbert space. However, by passing 
from $\varphi$ to $|\varphi|$, one might loose some desired properties of $\varphi$. 
For instance, if $\varphi$ is differentiable, $|\varphi|$ does not nessarily have to be so. 
Therefore we prefer to work in the Krein space. 

On an abstract level, we may consider the multiplication by $\varphi$ as an operator 
on $\L_2(\Omega,\mu)$, 
\[ \label{Wphi}
(W_\varphi f)(x):=\varphi(x)\hs f(x). 
\]
The operator $W_\varphi$ is unbounded, if $\mathrm{ess\;sup}\hs |\varphi | =\infty$, and self-adjoint, 
if we set $\dom(W_\varphi) :=\{ f\in \L_2(\Omega,\mu) : \int |f|^2 |\varphi|^2\hs \dd\mu<\infty\}$. 
With $\ip{\cdot}$ denoting the standard inner product on $\L_2(\Omega,\mu)$, Equation \eqref{L2} reads 
$$
[f,g] =\ip[f]{W_\varphi\hs g}, \quad f,g\in \dom(W_\varphi). 
$$ 
Furthermore, $\mathrm{ess\;inf}\hs  |\varphi | =0$ if and only if $0$ belongs to the spectrum of  $W_\varphi$. 
This is the general setting which will concern us in the next section.

\section{Frames in Hilbert spaces with $W$-metric}
\label{Sec-W}

The objective of this section is to show how to transfer a frame for a Hilbert space $\h$ to 
a non-regular Krein space $\h_W$. Although, on a technical level, we could treat the bounded 
and the unbounded case at once, we start by distinguishing between the following situations: 
First we suppose that the Gram operator $W$ is bounded to illustrate the completion process 
enforced by $0\in \spec(W)$. Second we allow $W$ to be an unbounded  
but assume that $0\notin \spec(W)$. 
The difference between these two cases is that in the latter, we can identify $\h_W$ 
with a subspace of $\H$, and in the former, we have to extend $\H$ by taking a completion. 
The general case will be obtained by combining both situations. 

Throughout this section, $(\h,\ip{\cdot}) $ stands for a separable Hilbert space,  and 
$W$ denotes a self-adjoint operator with domain $\dom(W)\subset \h$ 
and polar decomposition $W= J\hs |W|$.  
We assume that $\ker(W)=\{0\}$. 
Then $J$ is  a unitary self-adjoint operator. 
The letter $E$ will be used for the projection-valued measure on the 
Borel $\sigma$-algebra $\Sigma(\R)$ such that 
\[ \label{WE} 
        W= \int \lambda \, \dd E(\lambda). 
\]
Analogous to  \eqref{ipW}, we define 
\[ \label{fg}
        [f,g]:= \ip[f]{W\hs g},\quad f,\, g \in \dom(W). 
\] 
In this way, $\dom(W)$ becomes a decomposable non-degenerate inner product space 
with fundamental decomposition 
$\dom(W) = \D_+\oplus \D_-$ and fundamental symmetry $J$, where 
\[
\D_+\hsp:=\hsp E(0,\infty)\hs\dom(W),\ \ \D_-\hsp:=\hsp E(-\infty,0)\hs\dom(W),\ \ 
J\hsp =\hsp E(0,\infty)  \hsp-\hsp E(-\infty,0). 
\]
Here, $\ker(W)=\{0\}$ is necessary since otherwise $\dom(W)$ would be degenerate. 
From $J^2 = 1$, the polar decomposition $W= J\hs |W|$ and Equation \eqref{fg}, it follows that 
\[  \label{Jip}
 [f,g]_J:= \ip[f]{|W|\hs g},\quad f,\, g \in \dom(W). 
\]
Taking the closure under the norm defined by $[\cdot\hs,\cdot]_J$ and 
extending  $J$ to the closure (without changing the notation), we obtain a Krein space 
$(\h_W, [\cdot\hs,\cdot])$ 
with fundamental symmetry $J$ and fundamental decomposition 
$\h_W = \h_+ \oplus \h_-$
such that $\D_+$ and $\D_-$ are dense in $\h_+$ and $\h_-$, respectively. 

Recall from the end of Section \ref{KS} that, for a regular 
Gram operator,
$\ip{\cdot}$ and $[\cdot\hs ,\cdot]_J$ define equivalent norms on $\h$. 
It follows from \eqref{fH} 
(for instance, by applying the restriction 
$\| |W|^{-1}\|^{-1}\leq \lambda_1 < \lambda_2\leq \|W\|$ in the next proof)
that equivalent norms admit the same set of frames.  
Hence, by Theorem \ref{prop1}, any frame for $(\h,\ip{\cdot}) $ yields one for 
$(\h_W, [\cdot\hs,\cdot])$ and vice versa. 
The next proposition tells us that the same cannot hold for Gram operators 
which are not regular or unbounded. 

\begin{proposition} \label{not}
Let $(\h_W, [\cdot\hs,\cdot])$ denote the Krein space described above. 
\begin{itemize}
\item[\textit{i)}] If \hs$W$ is bounded and $0\hsp\in\hsp \spec(W)$, then 
any frame $\{f_n\}_{n\in\N}$ for $(\h,\ip{\cdot})$ 
is not a frame for $(\h_W, [\cdot\hs,\cdot])$. 
\item[\textit{ii)}]
If \hs$W$ is unbounded, then any frame $\{f_n\}_{n\in\N}\hsp\subset\hsp \dom(W)$ for $(\h,\ip{\cdot})$ 
is not a frame for $(\h_W, [\cdot\hs,\cdot])$. 
\end{itemize}
\end{proposition}
\begin{proof} Using the spectral theorem, we write 
$|W|=\int_{[0,\infty)} \lambda\, \dd F(\lambda)$, 
where $F$ denotes the corresponding projection valued measure. 
Given  $\lambda_2>\lambda_1>0$ such that $F([\lambda_1, \lambda_2])\neq 0$, choose 
$h\in F([\lambda_1, \lambda_2])\h$ with $\|h\|=1$. Then $h$ belongs to the domain of 
$\sqrt{|W|}^{\hs -1}$ so that $g:=\sqrt{|W|}^{\hs -1} h$ is well defined.  
From \eqref{Jip}, we get 
\[  \label{gnorm}
\|g\|_J^2=\ip[g]{|W|g}=\ip[h]{h}=1.
\] 
Furthermore, by the spectral theorem,
\[ \label{ineq}
\lambda_1  \,\leq\,
\int_{[\lambda_1,\lambda_2]} \hspace{-14pt}\lambda \;\dd\ip[h\hs]{F(\lambda)\hs h}
\,=\,\|\sqrt{|W|}\hs h\|^2 \,=\, \|\hs|W| g\hs\|^2 
\,\leq\, \lambda_2. 
\]

Now let $\{f_n\}_{n\in\N}$ be a frame for $(\h,\ip{\cdot})$ with frame bounds $0<A\leq B<\infty$. 
Suppose that  $0\hsp\in\hsp \spec(W)$.  Since $0$ is not an eigenvalue of $W$, there exists for each 
$\lambda_2>0$ a $\lambda_1\in(0,\lambda_2)$ such that $F([\lambda_1, \lambda_2])\neq 0$. 
Using Equations \eqref{fH}, \eqref{gnorm} and \eqref{ineq}, we get for $g$ as above 
\[  \label{notB}
\sum_{n\in \N}\left|\left[f_{n},g\right]\right|^{2} = \sum_{n\in \N}\left|\ip[f_{n}]{Wg}\right|^{2} 
\leq B\hs \|Wg\|^2 = B\hs \|\hs |W|\hs g\hs\|^2 \leq B \lambda_2 \|g\|_J^2, 
\]
and taking the limit $\lambda_2 \ra 0$ shows that 
there cannot exist a lower frame bound satisfying Definition \ref{fK}. 

If the Gram operator $W$ is unbounded, then for each 
$\lambda_1>0$ we find a $\lambda_2>\lambda_1$ such that $F([\lambda_1, \lambda_2])\neq 0$.
Similarly to \eqref{notB}, we compute 
$$
\sum_{n\in \N}\left|\left[f_{n},g\right]\right|^{2} = \sum_{n\in \N}\left|\ip[f_{n}]{Wg}\right|^{2} 
\geq A\hs \|Wg\|^2 = A\hs \|\hs |W|\hs g\hs\|^2 \geq A \lambda_1 \|g\|_J^2, 
$$
and since $\lambda_1>0$ was arbitrary, an upper frame bound does not exist. 
\end{proof}

We will now show how to transfer frames for the Hilbert space $\h$ 
to frames for the Krein space $\h_W$. 
As outlined in the beginning of this section, we start by considering a bounded Gram operator. 
\begin{theorem}  \label{Wbounded}
Let $W$ be a bounded self-adjoint operator on 
the Hilbert space $\h$ such that $\ker(W)=\{0\}$. 
\begin{enumerate}
\item[\textit{i)}]
The inclusion is $\h \subset \h_W$ is an equality if and only if \,$0\notin \spec(W)$. 
\item[\textit{ii)}]
$
\sqrt{|W|} : \h \subset \h_W\lra \h, \ \,  h\longmapsto \sqrt{|W|}h
$
\,defines an isometric operator and its extension 
$$
U:=\overline{\sqrt{|W|}} : \h_W\lra \h
$$
is a unitary operator. 
\item[\textit{iii)}] $\{k_n\}_{n\in \N} \subset \h$ is a frame for the Hilbert space $\h$ 
with frame bounds $A\leq B$ if and only if 
$\{U^{-1}k_n\}_{n\in \N} \subset \h_W$ is a frame for the Krein space 
$\h_W$ with frame bounds $A\leq B$. 
\end{enumerate}
\end{theorem}
\begin{proof}
Item iii) is an immediate consequence of ii) since unitary operators transfer frames 
into frames with the same frame bounds. 

We next show ii). By \eqref{Jip}, we have for all $h\in\h$
\[  \label{hWh}
\left\langle\sqrt{|W|}\hs h\,, \sqrt{|W|}\hs h\right\rangle 
= \left\langle h\,,|W|\hs h\right\rangle = [h,h]_J. 
\]
Hence  $\sqrt{|W|} :\h \subset \h_W\lra \h$ is an isometry and can be extended 
to the closure $\h_W$ of $\h$. It only remains to verify that its extension 
$\overline{\sqrt{|W|}}$ is surjective. For this, it suffices to show that 
$\sqrt{|W|}\hs\h$ is dense in $\h$. 
Since $\ker(W)=\{0\}$, we have $(\sqrt{|W|}\hs\h)^{\bot}=\ker(\sqrt{|W|})=\{0\}$, 
from which the result follows. 

The equality $\h = \h_W$ for regular Gram operators $W$ has already been discussed 
before Proposition \ref{not}. 
Suppose now  that  $0\in \spec(W)$. 
As well known, the Hilbert space norm $\|\cdot\|= \sqrt{\ip{\cdot}}$ is stronger than 
$\|\cdot\|_J= \sqrt{[\cdot\hs,\cdot]_J}$ since 
$\| h\|_J^2=\ip[h]{|W|h}\leq \|W\| \hs\|h\|^2$ for all $h\in\h$. 
From Proposition \ref{not}.i), it follows that the norms are not equivalent 
because equivalent norms give rise to the same set of frames. 
Therefore $\h\neq\h_W$, since otherwise $\id : \h\ra\h_W$ would be a continuous bijection 
and would thus have a bounded inverse, contradicting the fact that the norms are not equivalent. 
\end{proof}

The following theorem treats the case when $W$ is unbounded and $0\notin \spec(W)$. 
In contrary to the above situation, the Krein space $\h_W$ can then be identified with a subspace of $\h$. 
\begin{theorem}  \label{Wunbounded}
Let $W: \dom(W) \lra \h$ be a self-adjoint operator on 
the Hilbert space $\h$ such that $0\notin \spec(W)$. Then 
\begin{enumerate}
\item[\textit{i)}] $\h_W$ can be identified with $\dom(\sqrt{|W|})\subset \h$. 
\item[\textit{ii)}] $ \sqrt{|W|}\, :\,  \dom(\sqrt{|W|} )=\h_W\,\lra\, \H$ \,is a unitary operator. 
\item[\textit{iii)}] $\{k_n\}_{n\in \N} \subset \h$ is a frame for the Hilbert space $\h$ 
with frame bounds $A\leq B$ if and only if 
$\{\sqrt{|W|}^{\hs -1}k_n\}_{n\in \N} \subset \h_W$ is a frame for the Krein space 
$\h_W$ with frame bounds $A\leq B$. 
\end{enumerate}
\end{theorem}
\begin{proof} 
For the proof of i), we will use the fact that a densely defined linear operator $T$ is closed 
if and only if 
its domain is closed with respect to the graph norm 
\mbox{$(\|\cdot\|^2 + \| T(\hs\cdot\hs)\|^2)^{1/2}$}  \cite[Theorem 5.1]{Wd}. 
As by assumption $0\notin \spec(W)$, 
there exists an $\epsilon >0$ such that $|W|\geq \epsilon$. 
Therefore, 
for all $h\in \dom(W)$, we have $ \ip[h]{|W|\hs h}\geq \epsilon\hs \ip[h]{h}$. 
Recall that $\| h\|_J^2 =\ip[h]{|W|\hs h}=  \|\sqrt{|W|}\hs h\|^2$ for all $h\in \dom(W)$. 
Thus 
$$
\|h\|_J^2 = \|\sqrt{|W|}\hs h\|^2\leq 
\|h\|^2 \hsp +\hsp \|\sqrt{|W|}\hs h\|^2
\leq (\epsilon^{-1}\hsp + \hsp 1) \hs \|\sqrt{|W|}\hs h\|^2
=(\epsilon^{-1}\hsp +\hsp 1)\hs \|h\|_J^2,  
$$
hence  the $J$-norm $\| \cdot \|_J$ 
and the graph norm of $\sqrt{|W|}$ 
are equivalent on its common domain 
of definition. As remarked above, 
$\dom(\sqrt{|W|})$ is closed with respect to 
graph norm
since a self-adjoint operator is always closed. 
From the equivalence of norms 
on the dense subspace $\dom(W)\subset \dom(\sqrt{|W|})$, 
it follows that $\dom(\sqrt{|W|})$ is closed with respect to 
(the extension of) the norm $\|\cdot\|_J$. 
Taking the closure of $\dom(W)$   
allows us therefore to identify $\H_W$ with $\dom(\sqrt{|W|})$. 

To show ii), note $\sqrt{|W|}$ has a bounded inverse 
$\sqrt{|W|}^{-1}: \h\lra  \dom(\sqrt{|W|})$  
since $0\notin \spec(W)$. From 
$$
\left[\sqrt{|W|}^{-1}\hs h,\sqrt{|W|}^{-1}\hs h\right]_J 
= \left\langle\sqrt{|W|}^{-1}\hs h\,,|W|\sqrt{|W|}^{-1}\hs h\right\rangle
=\ip[h]{h}
$$
for all $h\in \dom(\sqrt{|W|})$,
we conclude that $\sqrt{|W|}^{-1}$ and thus $\sqrt{|W|}$ are unitary. 
Now iii) follows from the unitarity of $\sqrt{|W|}^{-1}$ as in the previous theorem. 
\end{proof}
Our last theorem deals with the general situation, where 
 $W$ may be unbounded and $0$ may belong to $\spec(W)$. 
\begin{theorem}\label{V3}
Let $W: \dom(W) \lra \h$ be a self-adjoint operator on 
the Hilbert space $\h$ such that $\ker(W)=\{0\}$. Then 
\begin{enumerate}
\item[\textit{i)}]
$\dom(\sqrt{|W|}\hs)$ is complete in the norm $\|\cdot\|_J$ if and only if 
\,$0\notin \spec(W)$.
\item[\textit{ii)}]
$\sqrt{|W|}\, :\, \dom(\sqrt{|W|}\hs)\, \lra\, \h$ \,can be extended to a unitary operator 
$$
U:=\overline{\sqrt{|W|}}\, : \,\h_W\,\lra\, \h. 
$$
\item[\textit{iii)}] $\{k_n\}_{n\in \N} \subset \h$ is a frame for the Hilbert space $\h$ 
with frame bounds $A\leq B$ if and only if 
$\{U^{-1}k_n\}_{n\in \N} \subset \h_W$ is a frame for the Krein space 
$\h_W$ with frame bounds $A\leq B$. 
\item[\textit{iv)}] If $\{k_n\}_{n\in \N} \subset \dom(\sqrt{|W|}^{\hs -1})$ is a frame for 
the Hilbert space $\h$, then the frame $\{U^{-1}k_n\}_{n\in \N}$ for 
 the Krein space $\h_W$ is given by $\{\sqrt{|W|}^{\hs -1} k_n \}_{n\in \N}$.
\end{enumerate}
\end{theorem}
\begin{proof} The first part will be proven by reducing to the previous theorems. 
For this, we decompose $\h$ into the orthogonal sum 
$$
 \h=\h_1\oplus \h_2,\qquad  \h_1:= E([-1,1])\h , \qquad \h_2 := E(\hs\R\setminus [-1,1]\hs)\h, 
$$
where $E$ denotes the projection valued measure from \eqref{WE}. 
Since the spectral projections commute with $W$ and $J= E(0,\infty)  - E(-\infty,0)$, 
the Hilbert spaces $\h_1$ and $\h_2$ are 
invariant under the action of these operators: 
\begin{align*}
W_1:=W\hspace{-3pt}\upharpoonright_{\h_1}\, &: \,\h_1 \,\lra\, \h_1, &
W_2:=W\hspace{-3pt}\upharpoonright_{\h_2} \, &: \,\h_2\cap \dom(W)  \,\lra\, \h_2, \\
J_1:=J\hspace{-3pt}\upharpoonright_{\h_1}\, &: \,\h_1 \,\lra\, \h_1, &
J_2:=J\hspace{-3pt}\upharpoonright_{\h_2} \,&: \,\h_2  \,\lra\, \h_2. 
\end{align*}
Also, the inner products respect this decomposition, that is,  
$$
[h_1,h_2] = \ip[h_1]{Wh_2}=0,\quad [h_1,h_2]_J = \ip[h_1]{|W|h_2}=0,
$$
for all $h_1\in\h_1$, and $h_2 \in \h_2\cap \dom(W)$. Therefore we can consider the 
closures of $\h_1$ and $\h_2\cap \dom(W)$ under the norm $\|\cdot\|_J$ separately. 
Note that, by  the spectral theorem, $\dom(W)= \h_1\oplus (\h_2\cap \dom(W))$. 
Taking the norm closure gives 
$$
    \h_W= \h_{W_1} \oplus \h_{W_2}, 
$$
where  $(\h_{W_i}, [\cdot\hs,\cdot])$ denotes the Krein space constructed from 
the Gram operator $W_i$ on $\h_i$, \,$i=1,2$.  
Applying Theorem \ref{Wbounded}.ii) and Theorem \ref{Wunbounded}.ii), 
we obtain a unitary operator 
\[ \label{U}
U:= \overline{\sqrt{|W_1|}}\oplus \sqrt{|W_2|} \ :\ 
\h_W=\h_{W_1} \oplus \h_{W_2}  \  \xrightarrow{\phantom{----}} \ \h=\h_1\oplus \h_2 
\]
such that its restriction to 
$\h_1\oplus\big(\h_2\cap \dom(\sqrt{|W|}\hs)\big)=\dom(\sqrt{|W|}\hs) $ 
is given by $\sqrt{|W_1|}\oplus \sqrt{|W_2|} = \sqrt{|W|}$. 
By the density of  $\dom(\sqrt{|W|}\hs)$ in $\h_W$, the unitary operator $U$ in \eqref{U} 
is the unique extension of $\sqrt{|W|}$. This proves ii). 

Now,  i) follows from Theorem \ref{Wbounded}.i) applied to 
$\h_1\subset \dom(\sqrt{|W|}\hs)$ and the bounded Gram operator 
$W_1=W\hspace{-3pt}\upharpoonright_{\h_1}$. Item iii) is a direct 
consequence of the unitarity of $U$ in ii), and iv) follows from 
$U\hspace{-3pt}\upharpoonright_{\dom(\sqrt{|W|})}\hs=\hs\sqrt{|W|}$,  
as observed in the previous paragraph. 
\end{proof}

By Theorem \ref{V3}.i),  if $0\in \spec(W)$, then the completion process will always require to 
add ``abstract'' elements to $\dom(\sqrt{|W|})$. 
One of the standard procedures is to view these elements as equivalence 
classes of Cauchy sequences. However, for $\L_2$-spaces and multiplication operators   
$W_\varphi$ as defined in \eqref{Wphi}, one can give an explicit description 
of $\h_{W_\varphi}$ in terms of measurable functions. 
In our final example, we illustrate this for $\h=\L_2(\R,\mu)$. 
In a certain sense, this is the generic case since 
any self-adjoint operator on a separable Hilbert space is 
unitarily equivalent to a direct sum of multiplication operators 
\cite[Theorem VII.3]{RS}.

\begin{example}
Let $\mu$ be a locally finite Borel measure on $\Sigma(\R)$ and let  $\F(\R)$ 
denote the space of all Borel measurable complex functions on $\R$. 
For $\varphi \in \F(\R)$, we define as in the end of the last section 
\begin{align*}
&\dom(W_\varphi):=\{ f\in \L_2(\R,\mu) : \int |f|^2\hs  |\varphi |^2 \hs \dd\mu <\infty\}, \\
& W_\varphi : \dom(W_\varphi) \lra \L_2(\R,\mu),\ \ (W_\varphi f)(t):= \varphi(t)\hs f(t). 
\end{align*}
Assume that $\varphi$ is real-valued and that $\mu(\{x\in\R:\varphi(x)=0\})=0$. 
Then $W_\varphi^*=W_\varphi$ and $\ker(W_\varphi)=\{0\}$ so that the assumptions 
of Theorem \ref{V3} are satisfied. Up to unitary equivalence, we may write 
\begin{align*}
&\!\!\!\! \h_{W_\varphi}=\left\{ f\in \F(\R) : \int |f|^2\hs  |\varphi |\hs \dd\mu <\infty\right\}, \\
 U=W_{|\varphi|^{1/2}} &\,:\, \h_{W_\varphi} \lra \L_2(\R,\mu),\ \ 
(W_{|\varphi|^{1/2}} f)(t):= \sqrt{|\varphi(t)|}\hs f(t), \\[2pt]
 U^{-1} =W_{|\varphi|^{-1/2}} &\,:\, \L_2(\R,\mu)\lra \h_{W_\varphi}, \ \ 
(W_{|\varphi|^{-1/2}} f)(t):= \mbox{$\frac{1}{\sqrt{|\varphi(t)|}}$}\hs f(t). 
\end{align*}
By Theorem \ref{V3}, 
any frame $\{ f_n\}_{n\in\N}$ for $\L_2(\R,\mu)$ 
determines a frame for the Krein space $\h_{W_\varphi}$ 
with the same frame bounds. It can be given by 
$\{\frac{1}{\sqrt{|\varphi|}} \hs f_n\}_{n\in\N}\subset \h_{W_\varphi}$. 
\end{example}

\mbox{ }\\
{\bf Acknowledgement.} 
We thank Francisco Mart\'inez Per\'ia for pointing out 
reference \cite{GMMMa} to us.

\bibliographystyle{amsplain}

\begin{thebibliography}{99}

\bibitem{A}
T. Ya.~Azizov and I.~S. Iokhvidov, 
\textit{Linear operator in spaces with an indefinite metric},  
Wiley-Interscience, Chichester, 1989.

\bibitem{B}
J. Bogn\'ar, 
\textit{Indefinite inner product spaces}, 
Springer Verlag, Berlin-Heidelberg, 1974. 

\bibitem{Chr} 
O. Christensen, 
\textit{An Introduction to Frames and Riesz Bases}, 
Birkh\"{a}user, Boston, 2003. 

\bibitem{Dau} 
I. Daubechies, 
\textit{The wavelet transform, time-frequency localization and signal analysis}, 
IEEE Trans. Inform. Theory \textbf{36} (1990), 961--1005. 

\bibitem{DGM} 
I. Daubechies, A. Grossmann and Y. Meyer, 
\textit{Painless nonorthogonal expansions}, 
J. Math. Phys. \textbf{27} (1986), 1271--1283. 

\bibitem{DS} 
R. J. Duffin and A. C. Schaeffer, 
\textit{A class of nonharmonic Fourier series}, 
Trans. Amer. Math. Soc. \textbf{72} (1952), 341--366. 

\bibitem{GMMM}
 J. I. Giribet, A. Maestripieri, F. Mart\'inez Per\'ia and P. Massey,  
\textit{On frames for Krein spaces},  
J. Math. Anal. Appl. \textbf{393} (2012), 122--137.

\bibitem{GMMMa} 
J. I. Giribet, A. Maestripieri, F. Mart\'inez Per\'ia and P. Massey, 
\textit{On a family of frames for Krein spaces}, 
	 arXiv:1112.1632v1. 

\bibitem{G}	 
K. Gr\"ochenig,
\textit{Foundations of time-frequency analysis},  
Birkh\"auser, Boston, 2001.	 

\bibitem{E} 
K. Esmeral, 
\textit{Marcos en espacios de Krein}, Master's Thesis, Universidad Michoacana, Morelia, 2011. 

\bibitem{RS} 
M. Reed and B. Simon, 
\textit{Methods of modern mathematical physics. I. Functional analysis}, 
Academic Press, New York, 1980. 

\bibitem{PW}
I. Peng and S. Waldron, 
\textit{Signed frames and Hadamard products of Gram matrices}, 
Linear Algebra Appl. \textbf{347} (2002), 131--157.

\bibitem{Wd}
J. Weidmann, 
\textit{Linear Operators in Hilbert Spaces},
Springer Verlag, New York, 1980. 

\end{thebibliography}

\end{document}